\documentclass[11pt]{amsart}
\usepackage{amsmath}
\usepackage{amssymb}
\usepackage{tikz-cd}
\usepackage{tikz}
\usepackage[utf8]{inputenc}
\usepackage[T1]{fontenc}
\usepackage{mathtools}
\usepackage{fixltx2e}
\usepackage{amsthm}
\usepackage{float}

\usepackage{hyperref}

\newtheorem{thm}{Theorem}[section]
\newtheorem{lemma}[thm]{Lemma}
\usepackage{cite}
\usepackage{subcaption}
\usepackage{caption}
\usepackage{graphicx}

\newtheorem{definition}{Definition}

\newtheorem{remark}{Remark}

\graphicspath{ {./images/} }

\DeclareMathOperator{\arcsinh}{arcsinh}

\newcommand{\nd}{\partial_\nu}

\title[Surfaces with first eigenvalue of large multiplicity]{\textbf{Constructing surfaces with first Steklov eigenvalue of arbitrarily large multiplicity}}
\author{Samuel Audet-Beaumont}
\address{D\'epartement de math\'ematiques et de statistique, Pavillon Alexandre-Vachon, Universit\'e Laval, Qu\'ebec, QC, G1V 0A6, Canada}
\email{samuel.audet-beaumont.1@ulaval.ca}

\begin{document}

\maketitle
\begin{abstract}
  We construct surfaces with arbitrarily large multiplicity for their first non-zero Steklov eigenvalue. The proof is based on a technique by M. Burger and B. Colbois originally used to prove a similar result for the Laplacian spectrum. We start by constructing surfaces $S_p$ with a specific subgroup of isometry $G_p:= \mathbb{Z}_p \rtimes \mathbb{Z}_p^*$ for each prime $p$. We do so by gluing surfaces with boundary following the structure of the Cayley graph of $G_p$. We then exploit the properties of $G_p$ and $S_p$ in order to show that an irreducible representation of high degree (depending on $p$) acts on the eigenspace of functions associated with $\sigma_1(S_p)$, leading to the desired result. 
\end{abstract}

\section{Introduction}
Let $(S,g)$ be a smooth, compact, connected, Riemannian surface with boundary $\partial S$. The Steklov eigenvalue problem consists of finding $(f,\sigma)\in C^\infty(S)\times \mathbb{R}$ with $f\not=0$ satisfying
$$  \begin{cases}
       \Delta f=0  &\quad\text{in } S,\\
       \nd f=\sigma f &\quad\text{on } \partial S;
\end{cases}$$
where $\nd$ is the outward normal derivative on $\partial S$. The numbers $\sigma$ are called the Steklov eigenvalues of $(S,g)$. They form a sequence $0=\sigma_0< \sigma_1 \leq \sigma_2\leq ... \nearrow \infty$, where each eigenvalue is repeated according to its multiplicity. This problem has attracted substantial attention in recent years. See~\cite{CGGS2024} for a survey of recent results and open problems. 

The focus of this paper is on the the multiplicity $m_1(S,g)$ of the first non-zero eigenvalue $\sigma_1>0$. Upper bounds for f $m_1(S,g)$ have been obtained by many authors in recent years. Notably, the works of Fraser and Schoen \cite{FS}, Jammes \cite{PJ} as well as Karpukhin, Kokarev and Polterovich \cite{KKP} brought multiple results concerning upper bounds for the multiplicity of Steklov eigenvalues in terms of the topology and the number of boundary components of the surface. The following result is a particular case of \cite[Theorem 1]{KKP}.
\begin{thm}
  Let $(S,g)$ be an  orientable compact Riemannian surface with non-empty boundary.
  Let $\mathfrak{g}$ be the genus of $(S,g)$. Then,
  $$m_1(S,g)\leq 4\mathfrak{g}+3.$$
\end{thm}
All known general upper bounds on $m_1(S,g)$ involve either the genus of the surface $S$, or both the genus and the number of boundary components of the surface $S$. These bounds are linear in term of the genus (and of the number of boundary components, when applicable). However, none of these bounds are known to be sharp for all genus and the question to know if surfaces with arbitrarily large multiplicity $m_1(S,g)$ is left open in the literature. The main result of this paper gives a positive answer to that question.
\begin{thm}\label{thm:main}
For each prime numer $p$, there exists a compact Riemannian surface $(S_p,g_p)$ of genus $\mathfrak{g}_p=1+p(p-1)$ with $2p$ boundary components such that $m_1(S_p,g_p)\geq p-1$.
\end{thm}
The proof of this result is based on a technique that was developped by M. Burger and B. Colbois \cite{BurgCol, BC} used originally to obtain a similar result for the eigenvalues of the Laplacian on closed surfaces. The idea is to construct a surface $(S_p,g)$ with a specific subgroup of isometry $G_p:=\mathbb{Z}_p \rtimes \mathbb{Z}_p^*$ by building the surface around the Cayley graph of $G_p$. Using the fact that the problem is invariant by isometry, we can define an action of $G_p$ on the space of eigenfunctions $E_1(S_p)$ associated to $\sigma_1$. This action on $E_1(S_p)$ is also isometric for the inner product of $L_2(S_p)$. This action can then be decomposed in irreductible representations of the group $G_p$. By choosing a group with an irreductible representation of high degree and by showing that this representation acts on $E_1(S_p)$, we show that $E_1$ must be of large dimension, which implies that $\sigma_1$ has high multiplicity. 

\subsection{Multiplicity bounds for the Steklov problem with density}
In his paper \cite{PJ}, Jammes studied a more general version of the Steklov problem:
$$  \begin{cases}
       \text{div}(\gamma \nabla f)=0  &\quad\text{in } S,\\
      \gamma  \nd f=\sigma \rho f.&\quad\text{on } \partial S.
     \end{cases}$$
Here, $\gamma\in C^\infty(S)$ and $\rho\in C^\infty(\partial S)$ are strictly positive density functions. The multiplicity of the first non-zero eigenvalue is now written $m_1(S,g,\gamma,\rho)$ to stress the dependence on the density functions.
Given a surface $S$ of genus $\mathfrak{g}$ with $l$ boundary components, let
$$M_1(\mathfrak{g},l)=\sup_{g,\rho,\gamma}m_1(S,g,\gamma,\rho)$$
where the supremum is over all Riemannian metrics and densities. 
In his paper~\cite{Jammes2014}, Jammes proposed the following conjecture: 
$$M_1(\mathfrak{g},l)\geq \text{Chr}_0(S)-1,$$
where the \emph{relative chromatic number} $\text{Chr}_0(S)$ is the number of vertices of the largest complete graph that can be embedded in $S$ with all vertices on the boundary $\partial S$. The conjecture is analogous to a conjecture of Colin de Verdière for closed surfaces~\cite{CdV1986}.
In his paper \cite{PJ}, Jammes proved part of his conjecture. He showed that
$$M_1(\mathfrak{g},l)\geq \text{Chr}_0(S)-1.$$
He also showed that $\text{Chr}_0(S)=O(\sqrt{\mathfrak{g}})$ as $\mathfrak{g}\to+\infty$, providing examples where the multiplicity $m_1(S,g,\gamma,\rho)$ is arbitrarily large. However, the proof of this last inequality rely heavily on the theory of perturbation and uses the freedom afforded by the density functions in an essential way. 


\textbf{Outline of this paper.} In section 2, we describe the construction of Cayley graphs $\Gamma_p$ associated with the groups $G_p$. In section 3, we discuss the construction of $S_p$ by using $G_p$ and discuss the basic properties of such surfaces. In section 4, we catalog all irreduscible representations of $G_p$. In section 5, we prove the main theorem.

\section{Construction of the Cayley graphs $\Gamma_p$} Let $p\geq 3$ be a prime number.
 $$G_p:= \mathbb{Z}_p \rtimes \mathbb{Z}_p^*$$
 where $\mathbb{Z}_p$ is the cyclic group of $p$ elements and $\mathbb{Z}_p^*$ is the multiplicative group modulo $p$. Alternatively, we can represent $G_p$ in terms of generators. Let $k$ be an integer such that $k$ is a primitive $p-1$-root modulo $p$. Then
 
 $$G_p= \langle \delta_1,\delta_2 \;|\;\delta_1^p=e,\; \delta_2^{p-1}=e,\;\delta_2^{-1}\delta_1\delta_2=\delta_1^k\rangle.$$
 One can verify that these two groups are isomorphic by considering the isomorphism $\psi$ that goes from $\mathbb{Z}_p \rtimes \mathbb{Z}_p^*$ to $G_p$ as written above induced by 
 $$\psi(\bar{1},\bar{1})=\delta_1,$$
 $$\psi(\bar{0},\bar{k})=\delta_2.$$
 We define $\Gamma_p$ as the Cayley graph of $G_p$ with generators $\{\delta_1,\delta_2\}$.
 
 \begin{definition} The Cayley graph $\Gamma(G,S)$ of a group $G$ with set of generator $S$ is the oriented graph with vertices $g\in G$ and edges between $g_1,g_2\in S$ if and only if there is $\delta \in S$ such that $g_1 \delta = g_2$. This relation also induces a natural orientation on each edge of the graph.
 \end{definition}
 Note that the construction of the graph is dependent on the choice of set of generators $S$. Also, since the vertices are associated with group elements, the number of vertices is the same as the number of elements in the group. Finally, since every generator and its inverse can be applied to an element $g\in G$, the degree of each vertex in $\Gamma$ is equal to $2|S|$. For more details on Cayley graphs, see \cite{Cayley}.
 
For example, we can use this definition to build the Cayley graph of $G_3$ using $S=\{\delta_1,\delta_2\}$ defined as before.

\begin{center}
\begin{tikzcd}
                                                      & \delta_1 \arrow[rd, two heads] \arrow[dd, bend right]         &                                                                \\
e \arrow[ru, two heads] \arrow[dd, bend right]        &                                                               & \delta_1^2 \arrow[ll, two heads] \arrow[dd, bend left]         \\
                                                      & \delta_1\delta_2 \arrow[ld, two heads] \arrow[uu, bend right] &                                                                \\
\delta_2 \arrow[rr, two heads] \arrow[uu, bend right] &                                                               & \delta_1^2\delta_2 \arrow[lu, two heads] \arrow[uu, bend left]
\end{tikzcd}

\end{center}
Here, the double arrows represent the action by $\delta_1$ and the simple arrows represent the action by $\delta_2$. Since $|G_p|= p(p-1)$, the graphs $\Gamma(G_p,\{\delta_1,\delta_2\})$ get increasingly complicated as $p$ grows larger. 

As a direct consequence of this construction, the graph $\Gamma(G_p,\{\delta_1,\delta_2\})$ admits an action of $G_p$ by isometry induced by the action of the elements of the group $G$ on the vertices of the graph by left multiplication. 

For the next sections, we will opt for the more compact notation $\Gamma_p:=\Gamma(G_p,\{\delta_1,\delta_2\})$.

\section{Construction of the surface $S_p$}
Let $\ell\in \mathbb{R}_{>0}$. Let 
$$\mathcal{S}:=\{(x,y)\in [-2,2]\times [-2,2] \; : \; x^2+y^2>1 \}$$
be the perforated square equipped with the euclidean metric of $\mathbb{R}^2$. Let $\mathcal{P}(\ell)$ be the pair of hyperbolic pants with one boundary component of length $2 \pi$ and two boundary components of length $\ell$. Note that by Theorem 3.1.7 of \cite{PB},  such pants are uniquely defined up to isometry by the lengths of their geodesic boundary components. The building block $B(\ell)$ is then obtained by gluing one boundary component of the cylinder $\mathcal{C}:= S^1\times [0,1]$  along the inner circular boundary component of $\mathcal{S}$ and the other boundary component of $\mathcal{C}$ along the boundary component of $\mathcal{P}(\ell)$ of length $2\pi$. The metric on $\mathcal{C}$ is chosen so as to interpolate smoothly between the metric on $\mathcal{S}$ and the metric on $\mathcal{P}(\ell)$. Note that every hyperbolic gluing in this section onwards is performed with no twist.
\begin{figure}[H]
\centering
\includegraphics[scale=0.3]{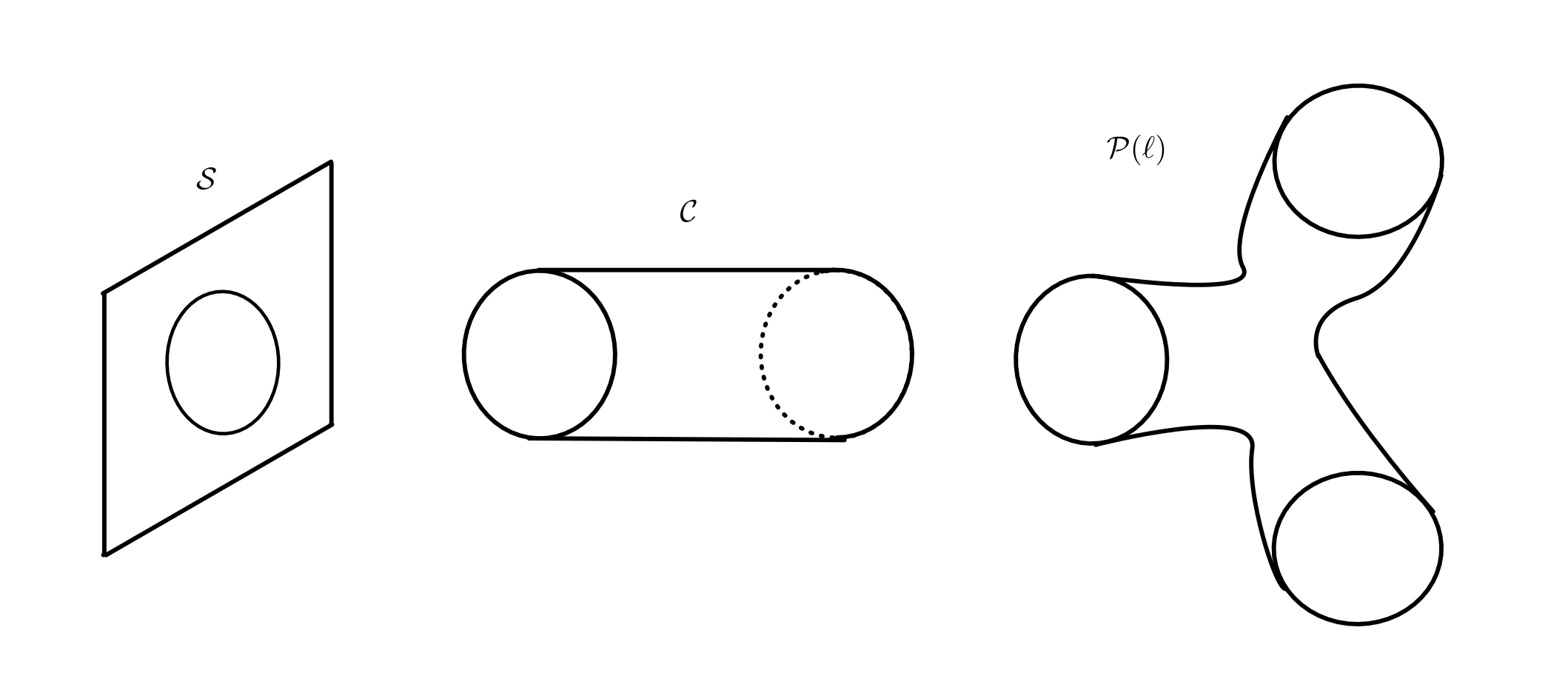}
\end{figure}
Thus the building block $B(\ell)$ pocesses $3$ boundary components : one of them corresponding to the sides of the square $\mathcal{S}$ and two of them corresponding to the boundary components of length $\ell$ in $\mathcal{P}(\ell)$. In what follows, we will denote by $\gamma_1^+$ and $\gamma_1^-$ the two boundary components of length $\ell$ in $B(\ell)$. We will also denote $(\gamma_2^+,\gamma_2^-)$ and $(b_1,b_2)$ the two pairs of parallel sides of length $4$ of the remaining boundary component in $B(\ell)$.
\begin{figure}[H]
\centering
\includegraphics[scale=0.4]{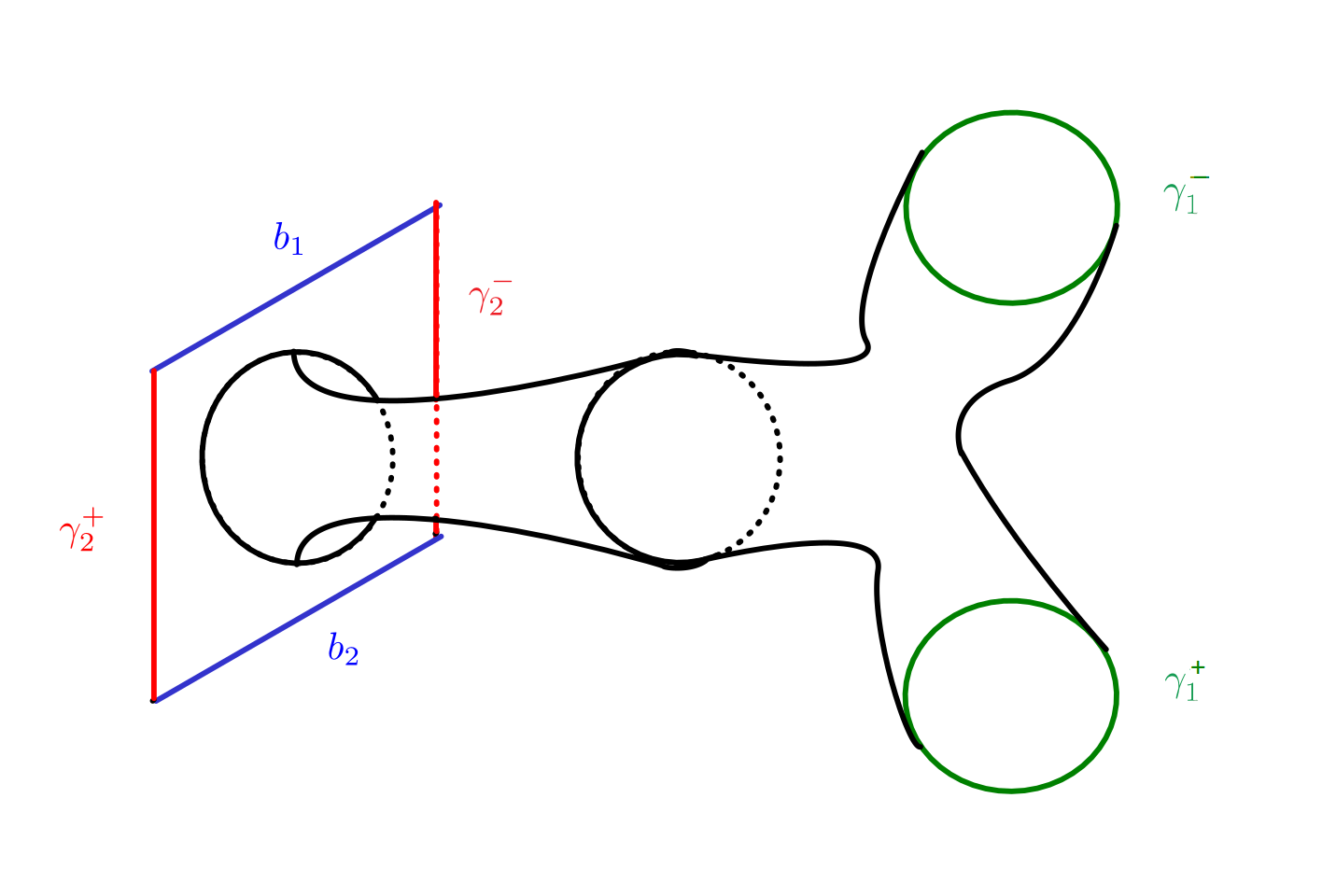}
	\caption{The building block $B(\ell)$}

\end{figure}
The surface $S_p(\ell)$ is obtained by gluing copies of $B(\ell)$ along the Cayley graph $\Gamma_p$ of section 2 via the following procedure. First, we assign to each vertex $v$ of $\Gamma_p$ a copy of $B(\ell)$ labeled $B_v(\ell)$. We will denote the boundary components $\gamma^\pm_i$ of $B_v(\ell)$ by $\gamma^\pm_{i,v}$ for $i\in \{1,2\}$. Let $v,w$ be two vertices of $\Gamma_p$ and let $B_v(\ell),B_w(\ell)$ be their respective assigned copy of $B(\ell)$. We glue $B_v(\ell)$ along $\gamma_{1,v}^+$ to $B_w(\ell)$ along $\gamma_{1,w}^-$ if there is an edge corresponding to $\delta_1$ in $\Gamma_p$ going from $v$ to $w$. Similarly, we glue $B_v(\ell)$ along $\gamma_{2,v}^+$ to $B_w(\ell)$ along $\gamma_{2,w}^-$ if there is an edge corresponding to $\delta_2$ in $\Gamma_p$ going from $v$ to $w$.

\begin{figure}[H]
\centering
\includegraphics[scale=0.2]{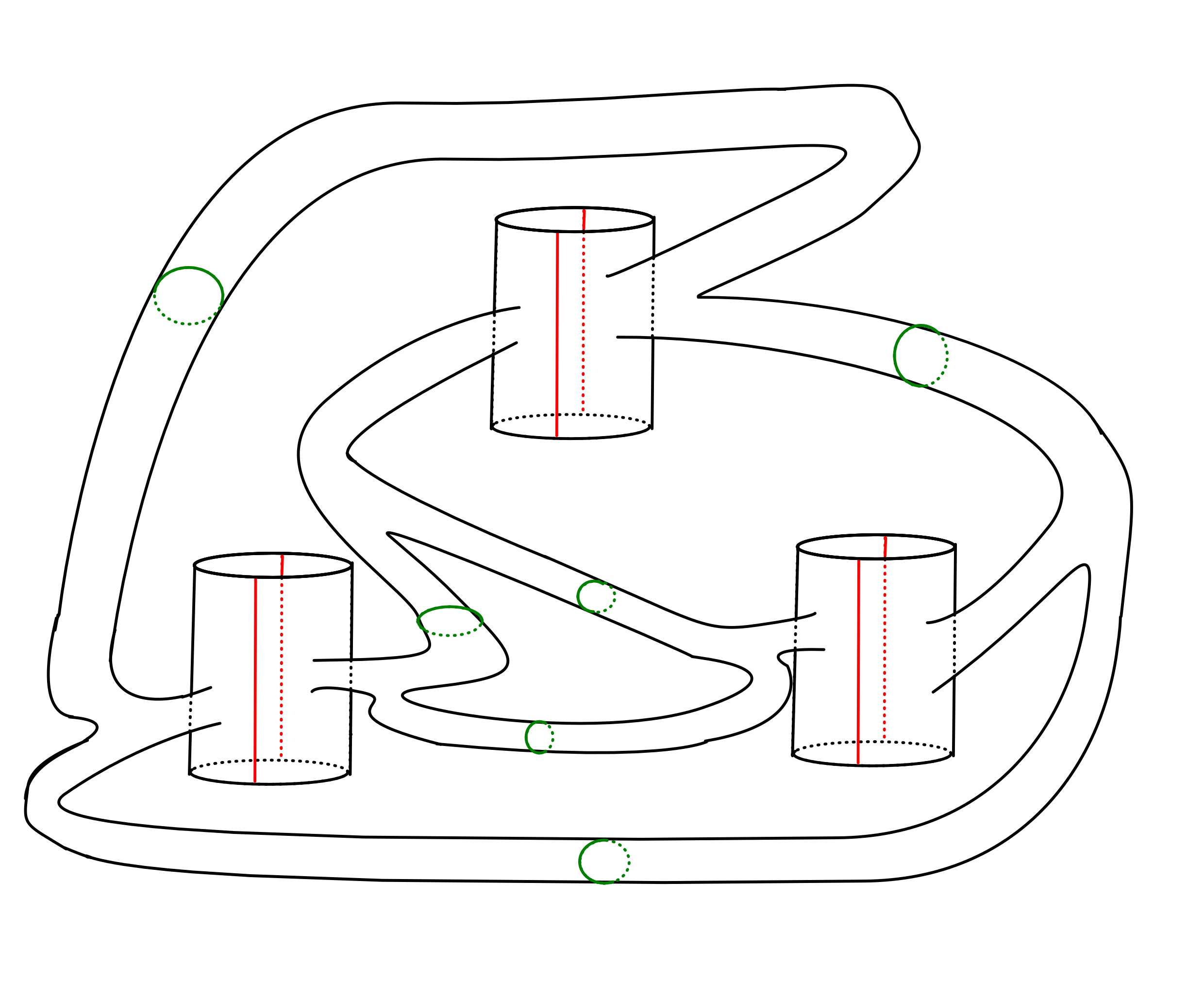}
	\caption{The surface $S_3$}

\end{figure}
\begin{lemma}
The surface $S_p$ has $2p$ boundary components and genus $g=1+|G_p|=1+p(p-1)$.
\end{lemma}
\begin{proof}
For $j\in\{0,1, \dots , p-1\}$, let 
$$C_j;=\{ g\in G \; : \; \exists m\in\mathbb{N}, \; g= \delta_1^j\delta_2^m\} $$
denote the $\delta_2$ generated cycles of $\delta_1^j$ in $\Gamma_p$. Since 
$$|G_p|\leq \sum_{j=0}^{p-1} |C_i|\leq p |C_p|=p(p-1)= |G_p|,$$
we have that the $C_i$ are disjoint and form a partition of $G$. Since, by construction, we have two boundary components per $\delta_2$ cycles $C_i$, the first assertion of the lemma follows.

For $j\in\{0,1, \dots , p-1\}$, let
$$\textit{Cyl}_j:= \bigcup_{g\in C_j}B_g(\ell) \subset S_p$$
be the region of $S_p$ corresponding to the $\delta_2$ cycles. It follows that the $\textit{Cyl}_j$ are pairwise disjoint cylinders with $2(p-1)$ handles for all $j\in\{0,1, \dots , p-1\}$. In order to get the genus of $S_p$, we glue disks along its boundary components and count the genus of the resulting closed surface $\hat{S}_p$. Note that
such gluings transform $\textit{Cyl}_j$ into spheres with $2(p-1)$ handles.
\begin{figure}[H]
\centering
\includegraphics[scale=0.25]{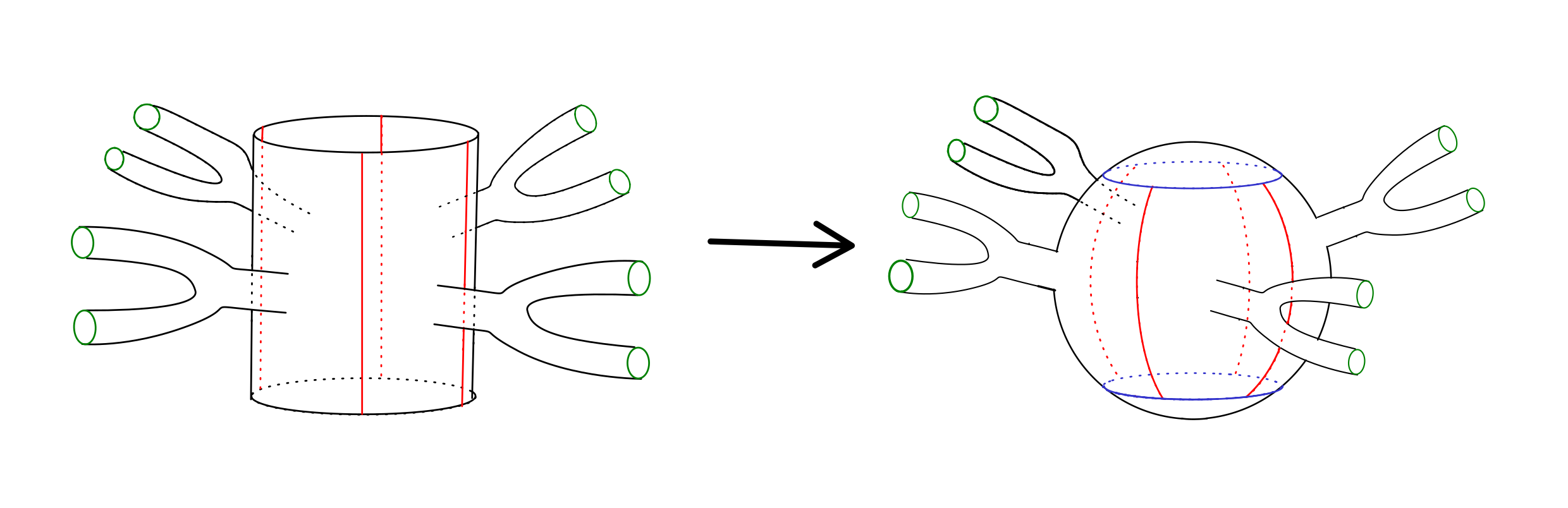}
	\caption{Transformation of $\textit{Cyl}_j$ in $S_5$ by gluing disks along its boundary components}

\end{figure}
Direct calculations also yield the following result.
\begin{lemma}
For all $j\in \{0,1,\dots , p-1\}$ and for all $g\in C_j$ we have either one of these two situations 
\begin{enumerate}
    \item $g\delta_1 \in C_{i+1}$ and $g \delta_1^{-1}\in C_{i-1},$
    \item $g\delta_1 \in C_{i-1}$ and $g \delta_1^{-1}\in C_{i+1}.$
\end{enumerate}
\end{lemma}
Thus $\hat{S}_p$ is homeomorphic to the surface composed of $p$ spheres $s_1,s_2,\dots s_p$ such that $s_i$ is linked to $s_j$ via $p-1$ disjoint cylinders if and only if $j\equiv i+1$ or $j\equiv i+1$ modulo $p$. One can represent these surfaces as a thickening of a $p$ cyclic graph of which every edge is replaced by $p-1$ edges. Such a graph can be seen below for $p=5$.

\begin{center}
\begin{tikzcd}
                                                                                                                                 &                                                                                                                                                             & s_4 \arrow[rrdd, no head, bend left] \arrow[rrdd, no head, bend left=49] \arrow[rrdd, no head] \arrow[rrdd, no head, bend right] &                                                                                                                           &                                                                                              \\
                                                                                                                                 &                                                                                                                                                             &                                                                                                                                  &                                                                                                                           &                                                                                              \\
s_5 \arrow[rruu, no head, bend left] \arrow[rruu, no head, bend right] \arrow[rruu, no head, bend left=49] \arrow[rruu, no head] &                                                                                                                                                             &                                                                                                                                  &                                                                                                                           & s_3 \arrow[ldd, no head, bend right] \arrow[ldd, no head, bend left=49] \arrow[ldd, no head] \\
                                                                                                                                 &                                                                                                                                                             &                                                                                                                                  &                                                                                                                           &                                                                                              \\
                                                                                                                                 & s_1 \arrow[rr, no head, bend left] \arrow[luu, no head, bend left=49] \arrow[luu, no head, bend left] \arrow[luu, no head, bend right] \arrow[luu, no head] &                                                                                                                                  & s_2 \arrow[ruu, no head, bend right] \arrow[ll, no head, bend left] \arrow[ll, no head, bend left=49] \arrow[ll, no head] &                                                                                             
\end{tikzcd}
\end{center}
The counting of the genus can then be handled directly by counting the bounded regions of this planar graph. The result follows. \end{proof}

\begin{remark}
The construction differs from the one presented in \cite{BC} due to the presence of boundary components. The goal here is to place the boundary components so as to minimize the number of boundary components and, more importantly, the genus of the resulting surface. By doing so, we can get a slightly stronger final result.
\end{remark}


\section{Irreducible representations of $G_p$}
Note that since $\Gamma_p$ is a Cayley graph of $G_p$, $G_p$ acts on the vertices $\Gamma_p$ by isometry via the group operation. This action can be lifted to an action by isometry of $G_p$ on $S_p$. Let 
$$\psi_{v\rightarrow w}\; : \;B_v(\ell)\rightarrow B_w(\ell)$$
be the canonical isometry between $B_v(\ell)$ and $B_w(\ell)$. Then one can explicit the above mentioned action by isometry of $G_p$ on $S_p$ via
$$\phi_g : S_p \rightarrow S_p,$$
$$x_v \mapsto \phi_g(x_v)= \psi_{v\rightarrow g(v)}(x_v),$$
where $x_v$ is a point of  $B_v(\ell)$. This action of $G_p$ on the surface $S_p$ induces an action of $G_p$ on the eigenspace of the functions associated with $\sigma_1(S_p)$ via the maps
$$\phi_g^*: E_1(S_p) \rightarrow E_1(S_p),$$
$$ f\mapsto f\circ \phi_g.$$
Since this action is linear, by taking a basis of $E_1(S_p)$, one can represent the action of $g\in G_p$  by a matrix $A_g\in GL_N(\mathbb{C})$, where $N$ represents the dimension of $E_1(S_p)$. Thus we have a representation of $G_p$ acting on $E_1(S_p)$. By changing the basis of $E_1(S_p)$ if necessary, the action of $G_p$ on $E_1(S_p)$ can be decomposed as a direct sum of irreducible representations. This induces a decomposition of $E_1(S_p)$ in a direct sum of subspaces associated to irreducible representations of $G_p$ acting on $E_1(S_p)$. Thus if we can show that an irreducible representation of high degree acts on $E_1(S_p)$, it follows that there is a subspace of $E_1(S_p)$ of high dimension, which in turn implies that $E_1(S_p)$ has high dimension and that $\sigma_1(S_p)$ has high multiplicity. All of the aforementioned results on representation theory can be seen in greater detail in \cite{BC}.

Thus we will now explicit the irreducible representations of $G_p$. Since it is a well known topic, we will omit the technical details and simply state the results.

First, $G_p$ admits $(p-1)$ irreducible representations of degree $1$. Let $\alpha$ be a primitive $(p-1)$-root of $1$ in $\mathbb{C}$. Then, for all $r$ such that $1 \leq r \leq p-1$, we have the following irreducible representations :
$$\rho_r : G_p\rightarrow \mathbb{C}^*$$
induced by
$$\rho_r(\delta_1)=1,$$
$$\rho_r(\delta_2)=\alpha^r.$$
One can verify that $\rho_r$ is a group homomorphism for all $r$. These representations are all irreducible and non-equivalent. Observe that since $\delta_1$ is in the kernel of $\rho_r$, the subgroup $\mathbb{Z}_p$ of $G_p$ generated by $\delta_1$ lies in the kernel of $\rho_r$ for all $r$.

Next, we will show that $G_p$ admits $1$ irreducible representation of degree $p-1$. Let $\xi$ be a primitive $p$-root of $1$ in $\mathbb{C}$ and $k$ be a primitive $(p-1)$-root of $1$ modulo $p$. It follows that $\xi^{k^i}\not=\xi^{k^j}$ if $j\not=i$ for all $0\leq i,j\leq p-2$. We have that 
$$\rho:G_p\rightarrow GL_{p-1}(\mathbb{C})$$
induced by
$$\rho(\delta_1)=
\begin{pmatrix}
\xi &  &  & 0\\
 & \xi^k &  & \\
 &  & ... & \\
0 &  &  & \xi^{k^{p-2}}
\end{pmatrix}, \;
\rho(\delta_2)=\begin{pmatrix}
0 & 0 & ... &  & 1\\
1 & 0 & & & 0\\
0 & 1 & & & \\
  & 0 & & & \\
   &  & & & \\
0 &  & & 1 & 0
\end{pmatrix}$$
constitutes a representation of $G_p$. We then prove the following lemma.
\begin{lemma}
$\rho$ is an irreducible reprensation of $G_p$. 
\end{lemma}
\begin{proof}
By Schur's lemma (Lemma 1.7 in \cite{representationtheory} for example), it suffices to prove that the only matrices commuting with $\rho(g)$ for all $g\in G_p$ are scalar matrices.

Let $A\in GL_{p-1}(\mathbb{C})$ such that it commutes with $\rho(g)$ for all $g\in G_p$. Then, in particular, it commutes with $\rho(\delta_1)$. Let
$$B:=A\rho(\delta_1),$$
$$C:= \rho(\delta_1) A.$$
It follows that $B=C$ since $A$ and $\rho(\delta_1)$ commute. For all $i\not=j$, we have :
$$b_{i,j}= a_{i,j}\xi^{k^{j-1}},$$
$$c_{i,j}= a_{i,j}\xi^{k^{i-1}}.$$
Since $\xi^{k^{j-1}}\not=\xi^{k^{i-1}}$ if $j\not=i$, it implies that $A$ must be diagonal. Let 
$$X:=A \rho(\delta_2),$$
$$Y:= \rho(\delta_2) A.$$
Once again, we obtain that $X=Y$ since $A$ must commute with $\rho(\delta_2)$. For all $1\leq i\leq p-2$, we have :
$$x_{i+1,i}=a_{i,i},$$
$$y_{i+1,i}=a_{i+1,i+1}.$$
It implies that
$$a_{1,1}=a_{2,2}=\cdots=a_{p-1,p-1}.$$
It follows that $A$ must be scalar and the proof is complete.

\end{proof}

Finally, since the sum of the square of the degrees of these irreducible representations is equal to $|G_p|$, we can conclude that we have all the possible irreducible representations of $G_p$.

\section{Proof of the main result}
What is left to show is that the representation of degree $p-1$ acts on $E_1(S_p(\ell))$ for $\ell$ small enough.  We will do so by showing that none of the irreducible representations of degree $1$ act on $E_1(S_p(\ell))$ for $\ell$ small enough. Our strategy will consist of considering the quotient of $S_p(\ell)$ by the action of the normal subgroup $\mathbb{Z}_p$ of $G_p$, that we will note $S'_p(\ell)$. Since $\mathbb{Z}_p$ lies in the kernel of every representation of degree $1$, if we can show that $\sigma_1(S_p(\ell))\not=\sigma_1(S'_p(\ell))$ for all $\ell<\epsilon$ for a certain $\epsilon$, then the result will follow.

In order to show that the first non-zero eigenvalues of $S_p(\ell)$ and $S'_p(\ell)$ do not coincide for $\ell$ small enough, we start by showing the following lemma.

\begin{lemma}\label{sigmaS_ppetit}
$$\lim_{\ell \rightarrow 0}\sigma_1(S_p(\ell))= 0.$$
\end{lemma}
To prove this lemma, we will need two classical results. 
\begin{thm}\label{vc}(Variational characterization of the Steklov eigenvalues)
$$\sigma_k=\min_{E\in \mathcal{H}_{k+1}}\max_{0\not=u\in E}R_{S}(u),$$
where $\mathcal{H}_{k+1}$ is the set of all $k+1$-dimensional subspaces in the Sobolev space $H^1(S)$ and 
$$R_{S}(u)=\frac{\int_{S}|\nabla u |^2 dA}{\int_{\partial S}|u|^2 ds}.$$
\end{thm}
The second important result is an application of the collar theorem (\cite{PB}, Theorem 4.1.1). Let
$$\mathcal{C}(S_p(\ell)):= \bigcup_{g\in G_p}\gamma_{1,g}^\pm  \subset S_p(\ell)$$
and the neighborhood of $x\in S_p(\ell)$
$$\mathcal{T}(x):=\{q\in S_p(\ell) \;|\; dist(q,x)\leq w(\ell)\}$$
where
$$w(\ell):= \arcsinh(\coth(\tfrac{\ell}{2})).$$
Since the $\gamma_{1,g}^\pm \subset S_p(\ell)$ are gluing points of hyperbolic pants of constant curvature $-1$ the collar theorem applies locally to $\mathcal{T}(\alpha)$ and we get the following :

\begin{thm}\label{collar} For connected components $\alpha,\beta\subset \mathcal{C}(S_p(\ell))$, then $\mathcal{T}(\alpha)\cap \mathcal{T}(\beta)=\emptyset$ if and only if $\alpha\cap \beta = \emptyset$. Furthermore, for all connected components $\alpha\in \mathcal{C}(S_p)$, $\mathcal{T}(\alpha)$ is isometric to the cylinder $[-w(\ell),w(\ell)]\times \mathbb{S}^1$ with the Riemannian metric $ds^2=d\rho^2+\ell^2 \cosh^2(\rho)dt^2$ on the Fermi coordinates of the cylinder based on $\alpha$ parametrized with speed $\ell$.
\end{thm}
We can now prove lemma \ref{sigmaS_ppetit} using theorem \ref{vc} and lemma \ref{collar}.
\begin{proof}
Our goal will be to build two test functions $f_1,f_2\in H^1(S_p(\ell))$ such that $R_{S_p(\ell)}(f_i)$ goes to zero as $\ell$ goes to zero. 

Observe that $S_p - \mathcal{C}(S_p(\ell))$ has $p$ connected components. This is a direct consequence of the construction of $S_p(\ell)$ around the Cayley graph of $G_p$ with generators $\delta_1,\delta_2$. We will note these connected components $M_i$ for $i \in \{1,2,\dots, p\}$ and $\mathcal{C}(M_i)\subset \mathcal{C}(S_p(\ell))$ the union of geodesics joining $M_i$ to its complement in $S_p(\ell)$. Thus $\mathcal{C}(M_i)$ is the union of $2p-2$ disjoint geodesics of lengths $\ell$.

Let
$$M_1^0:=M_1-\bigcup_{x\in \mathcal{C}(M_1)}\mathcal{T}(x)\cap M_1,$$
$$M_2^0:=M_2-\bigcup_{x\in \mathcal{C}(M_2)}\mathcal{T}(x)\cap M_2.$$
and
$$f_1(p):=
  \begin{cases}
       1 &\quad\text{if } q\in M_1^0\\
       0 &\quad\text{if } q\in M_1^c\\
       \frac{dist(q,\alpha)}{ w(\ell)} &\quad\text{if } q\in \mathcal{T}(\alpha)\cap M_1 \text{ for some connected component } \alpha\in \mathcal{C}(M_1)\\      
     \end{cases}$$
$$f_2(p):=
  \begin{cases}
       1 &\quad\text{if } q\in M_2^0\\
       0 &\quad\text{if } q\in M_2^c\\
        \frac{dist(q,\alpha)}{ w(\ell)} &\quad\text{if } q\in \mathcal{T}(\alpha)\cap M_2 \text{ for some connected component } \alpha\in \mathcal{C}(M_2)\\       
     \end{cases}$$
Since both functions have disjoint supports, $span\{f_1,f_2\}\subset H^1(S_p(\ell))$ is a subspace of dimension $2$. We can now calculate $R_{S_p(\ell)}(f_i)$. 
$$
R_{S}(f_1)=\frac{\int_{S_p(\ell)}|\nabla f_1 |^2 dV}{\int_{\partial S_p(\ell)}|f_1|^2 dS}
$$$$\leq 
\frac{\sum_{\alpha}\int_{\mathcal{T}(\alpha)\cap M_1}|\nabla f_1 |^2 dV}{A}$$
\begin{equation}\label{eq1}=\frac{(2p-2)}{A}\int_{\mathcal{T}(\hat{\alpha})\cap M_1}|\nabla f_1 |^2 dV
\end{equation}
where the sum is taken over all the connected components $\alpha\subset \mathcal{C}(M_1)$. Here, $A$ is a constant independent of $\ell$ and $\hat{\alpha}$ is a given connected component of $\mathcal{C}(M_1)$. This last integral yields

$$
\int_{\mathcal{T}(\hat{\alpha})\cap M_1}|\nabla f_1 |^2 dV=\frac{1}{w(\ell)^2}\int_{\mathcal{T}(\hat{\alpha})\cap M_1}|\nabla dist(q,\hat{\alpha}) |^2 dV
$$
\begin{equation}\label{eq2}=\frac{1}{w(\ell)^2}\int_{\mathcal{T}(\hat{\alpha})\cap M_1} dV=\frac{Area(\mathcal{T}(\hat{\alpha})\cap M_1)}{w(\ell)^2}. 
\end{equation}
We now need to calculate $Area(\mathcal{T}(\hat{\alpha})\cap M_1)=\tfrac{Area(\mathcal{T}(\hat{\alpha}))}{2}$. By lemma \ref{collar}, $\mathcal{T}(\hat{\alpha})$ is isometric to the cylinder $[-w(\ell),w(\ell)]\times \mathbb{S}^1$ with the Riemannian metric $ds^2=d\rho^2+\ell^2 \cosh^2(\rho)dt^2$ on the Fermi coordinates of the cylinder based on $\hat{\alpha}$ parametrized with speed $\ell$. Thus
$$
Area(\mathcal{T}(\hat{\alpha}))=\ell^2 \int_{0}^{1}\int_{-w(\ell)}^{w(\ell)}\cosh^2(\rho) d\rho dt$$
\begin{equation}\label{eq3} =\frac{\ell^2}{2}\int_{0}^{1} [\rho + \cosh(\rho)\sinh(\rho)]_{\rho=-w(\ell)}^{\rho=w(\ell)}dt =\ell^2 w(\ell).
\end{equation}
Piecing equations (\ref{eq1}), (\ref{eq2}) and (\ref{eq3})  together, we obtain the upper bound
$$R_{S}(f_1)\leq \frac{(2p-2)\ell^2}{4A w(\ell)}.$$
Since $w(\ell)$ goes to $\infty$ when $\ell$ goes to $0$, we have that
$$\lim_{\ell\rightarrow 0}R_{S}(f_1)=0. $$
The calculation for $f_2$ is similar. By the variational characterization of the Steklov eigenvalues, the result follows. 
\end{proof}
What remains to show is that $\sigma_1(S'_p(\ell)$ does not go to zero when $\ell$ goes to zero. Before proving this fact directly, we will discuss some basic properties of the compact surface
$$S'_p(\ell):= S_p(\ell) / \mathbb{Z}_p.$$
Since 
 $$G_p:= \mathbb{Z}_p \rtimes \mathbb{Z}_p^*$$
 it follows that 
 $$G_p/ \mathbb{Z}_p \cong \mathbb{Z}_p^*$$
via the isomorphism induced by the quotient map.
Hence $S'_p(\ell)$ is composed of copies of $B(\ell)$ that we will note $B_{\bar{\delta}_2^i}(\ell)$ for $1\leq i \leq p-1$ corresponding to the equivalence classes of elements $\delta_2^i$ in $G_p$ under the quotient. Observe that the $B_{\bar{\delta}_2^i}(\ell)$ will be glued together along $\gamma_{2,\bar{\delta}_2}^+$ and $\gamma_{2,\bar{\delta}_2}^-$ in $S'_p(\ell)$ in the same way that the $B_{\delta_2^i}(\ell)$ are glued together along $\gamma_{2,\delta_2^i}^+$ and $\gamma_{2,\delta_2^i}^-$ in $S_p(\ell)$.  In order to get a better idea of the properties of $S'_p(\ell)$, we will study how the  $B_{\bar{\delta}_2^i}(\ell)$  are glued together along $\gamma_{1,\bar{\delta}_2^i}^+$ and $\gamma_{1,\bar{\delta}_2^i}^-$.

Let $j$ be an integer such that $1\leq j \leq p-1$. Then there exist $\kappa$, an integer such that 
$$j\kappa\equiv -1 \mod p.$$
Observe that since $\bar{\delta_2^j}=\overline{\delta_1^\kappa \delta_2^j}$ then $B_{\bar{\delta_2^j}}(\ell)= B_{\bar{\delta_1^\kappa \delta_2^j}}(\ell)$. Note also that in $S_p(\ell)$, $\gamma_{1,\delta_1^\kappa \delta_2^j}^+$ is glued along $\gamma_{1,\delta_1^\kappa \delta_2^j \delta_1}^-$. Since we have that
$$\delta_1^\kappa \delta_2^j \delta_1=\delta_2^j \delta_1^{\kappa j+1}=\delta_2^j,$$
it follows that in the quotient $S'_p(\ell)$, $\gamma_{1,\bar{\delta}_2^j}^+$ is glued along $\gamma_{1,\bar{\delta}_2^j}^-$ for all $j$. All the gluings are done with no twist, since the original gluings had no twist.
\begin{figure}[H]
\centering
\includegraphics[scale=0.2]{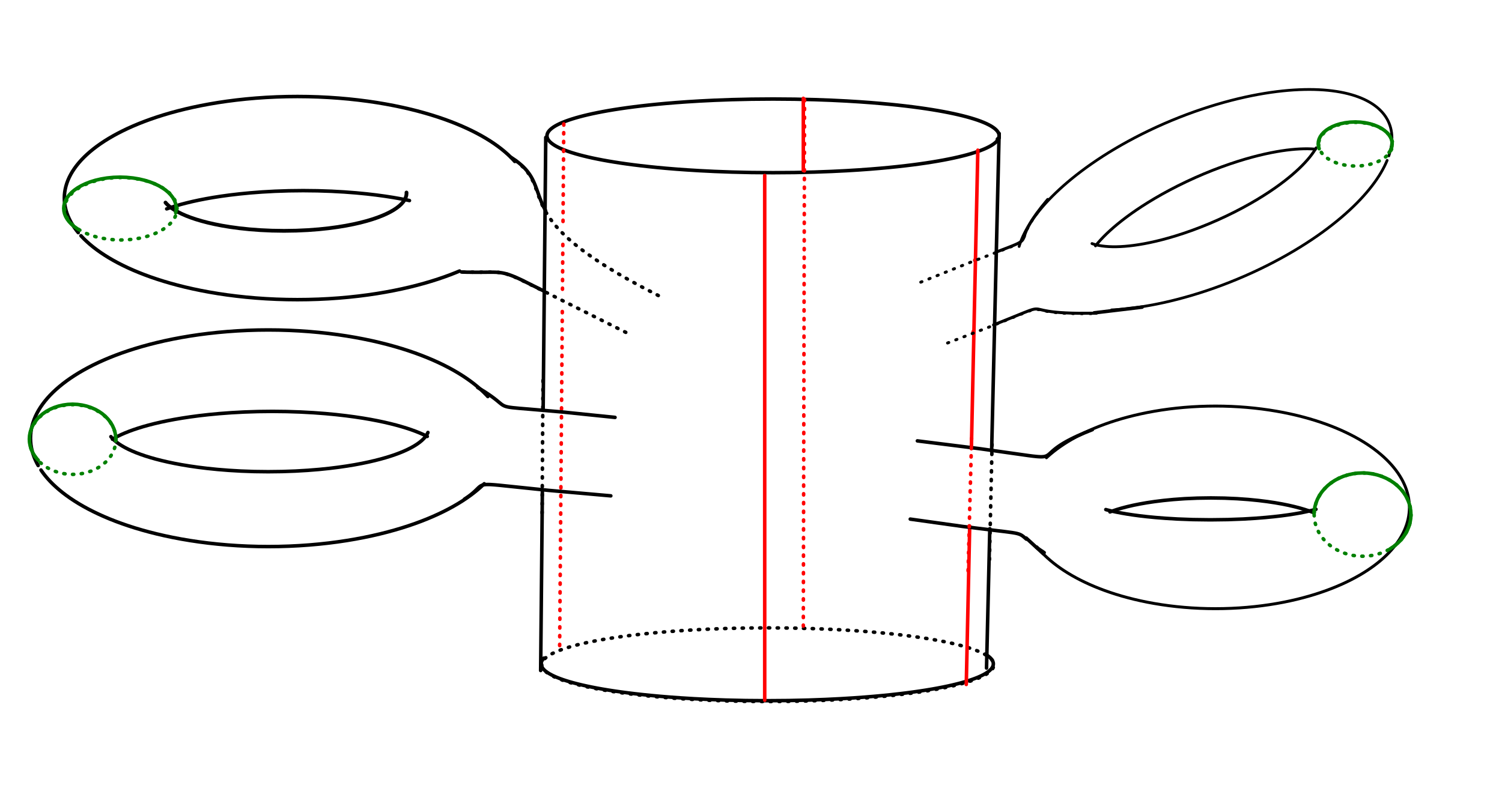}
	\caption{The surface $S'_5$}

\end{figure}
It follows that $S'_p(\ell)$ can be constructed in the following way. Let $B'(\ell)$ be the building block obtained from $B(\ell)$ by glueing $\gamma_1^+$ along $\gamma_1^-$. Then $S'_p(\ell)$ is the surface obtained by gluing the building blocks $B'(\ell)$ along $\gamma_2^\pm$ following the Cayley graph of $\mathbb{Z}_p^*$ with one generator.
\begin{figure}[H]
\centering
\includegraphics[scale=0.2]{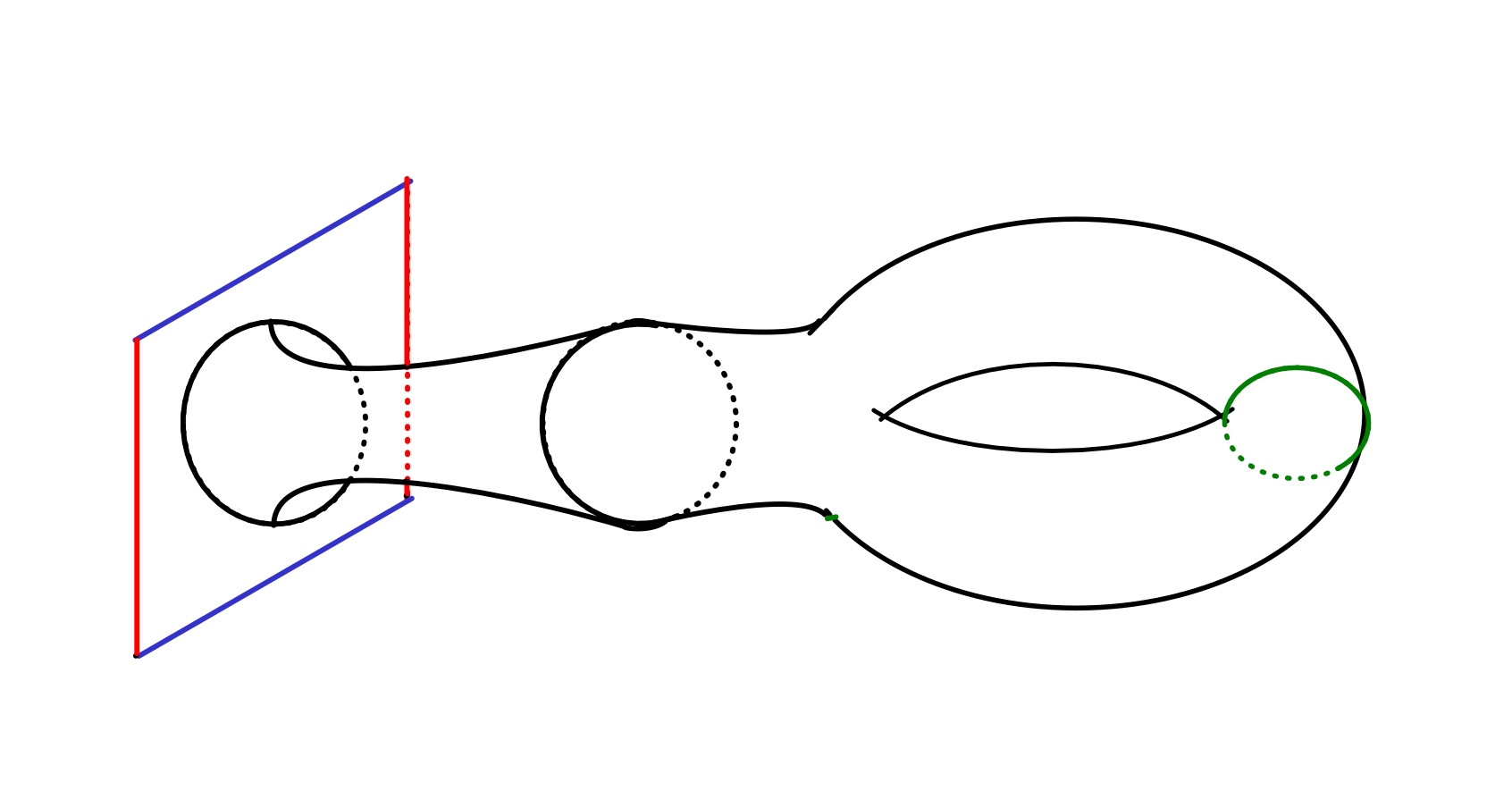}
	\caption{The building block $B'(\ell)$}
\end{figure}
A direct consequence of this discussion is that $S'_p(\ell)-\cup_{g\in G_p}\gamma_{1,\bar{g}}^\pm$ possesses a single connected component.
We will use this information in order to prove the following lemma and complete the proof of the main result.

  




\begin{lemma}\label{lowerb}
There exist a constant $C=C(p)>0$ such that for all $\ell>0$,
$$\sigma_1(S'_p(\ell))\geq C.$$
\end{lemma}

In order to prove this, we will need to introduce the mixed Steklov-Neumann problem. Let $A$ be a domain on the surface $S$ such that $\partial S \subset A$. We note $\partial_I A= \partial A -\partial S$ the interior boundary of $A$. The Steklov-Neumann problem on $A$ is the eigenvalue problem that consist of finding $\sigma^N\in \mathbb{R }$ and $f:\bar{A} \rightarrow \mathbb{R}$ non-zero such that:
$$  \begin{cases}
       \Delta f=0  &\quad\text{in } A,\\
       \nd f=\sigma^N f &\quad\text{on } \partial S,\\
       \nd f=0 &\quad\text{on } \partial_I A. \\           
     \end{cases}$$
Here $\nd$ is the outward normal derivative on $\partial A$. As in the Steklov problem, the eigenvalues of this mixed problem form a discrete sequence $0=\sigma_0^N(A) < \sigma_1^N(A) \leq ... \nearrow \infty$. We also have the following characterization.
\begin{thm}[Variational characterization of the Steklov-Neumann eigenvalues]\label{StekNew}
$$\sigma_k^N(A)=\min_{E\in \mathcal{H}_{k+1}}\max_{0\not=u\in E}R_{A}^N(u),$$
where $\mathcal{H}_{k+1}$ is the set of all $k+1$-dimensional subspaces in the Sobolev space $H^1(A)$ and 
$$R_{A}^N(u)=\frac{\int_{A}|\nabla u |^2 dV_{A}}{\int_{\partial S}|u|^2 dV_{\partial S}}.$$
\end{thm}
With this result in hand, we can prove lemma \ref{lowerb}.
\begin{proof}
By a direct comparison of $R_{A}^N$ and $R_{S}$, we obtain the inequality
$$\sigma_1^N(A)\leq \sigma_1(S)$$ for all suitable $A\subset S$. Thus it suffices to find a candidate $A\subset S'_p$ such that $A$ is invariant under change of $\ell$. We will start by identifying such a domain in the building block $B'(\ell)$ and then lift it to $S'_p(\ell)$.

From the discussion above, $B'(\ell)$ is obtained by gluing the euclidian punctured square $\mathcal{S}$ along a cylinder $\mathcal{C}$ and a pair of hyperbolic paints $\mathcal{P}(\ell)$ with one boundary component on length $1$ and two of length $\ell$. Thus, $\mathcal{S}$ when seen as a subset of $B'(\ell)$ is invariant under changes of $\ell$. Let
$$\mathcal{S}_{\bar{g}} \subset B_{\bar{g}}(\ell) \subset S'_p(\ell)$$
be the copy of $\mathcal{S}$ in $B_{\bar{g}}(\ell)$ for $\bar{g}\in \mathbb{Z}_p^*$. Let
$$A:= \bigcup_{\bar{g}\in \mathbb{Z}_p^*}\mathcal{S}_{\bar{g}} .$$
and 
$$\mathcal{C}(S'_p(\ell)):= \bigcup_{\bar{g}\in \mathbb{Z}_p^*}\gamma_{1,\bar{g}}^\pm  \subset S'_p(\ell).$$

Note that $A$ is still connected since $S'_p-\mathcal{C}(S'_p)$ is connected. We also have $\partial S'_p \subset A$. Furthermore, we can observe that $A$ is invariant under the change of $\ell$. Thus theorem \ref{StekNew} yields
$$0< \sigma_1^N(A) \leq \sigma_1(S'_p)$$
completing the proof.
\end{proof}

\begin{figure}[H]
\centering
\begin{subfigure}{.5\textwidth}
  \centering
  
\includegraphics[scale=0.18]{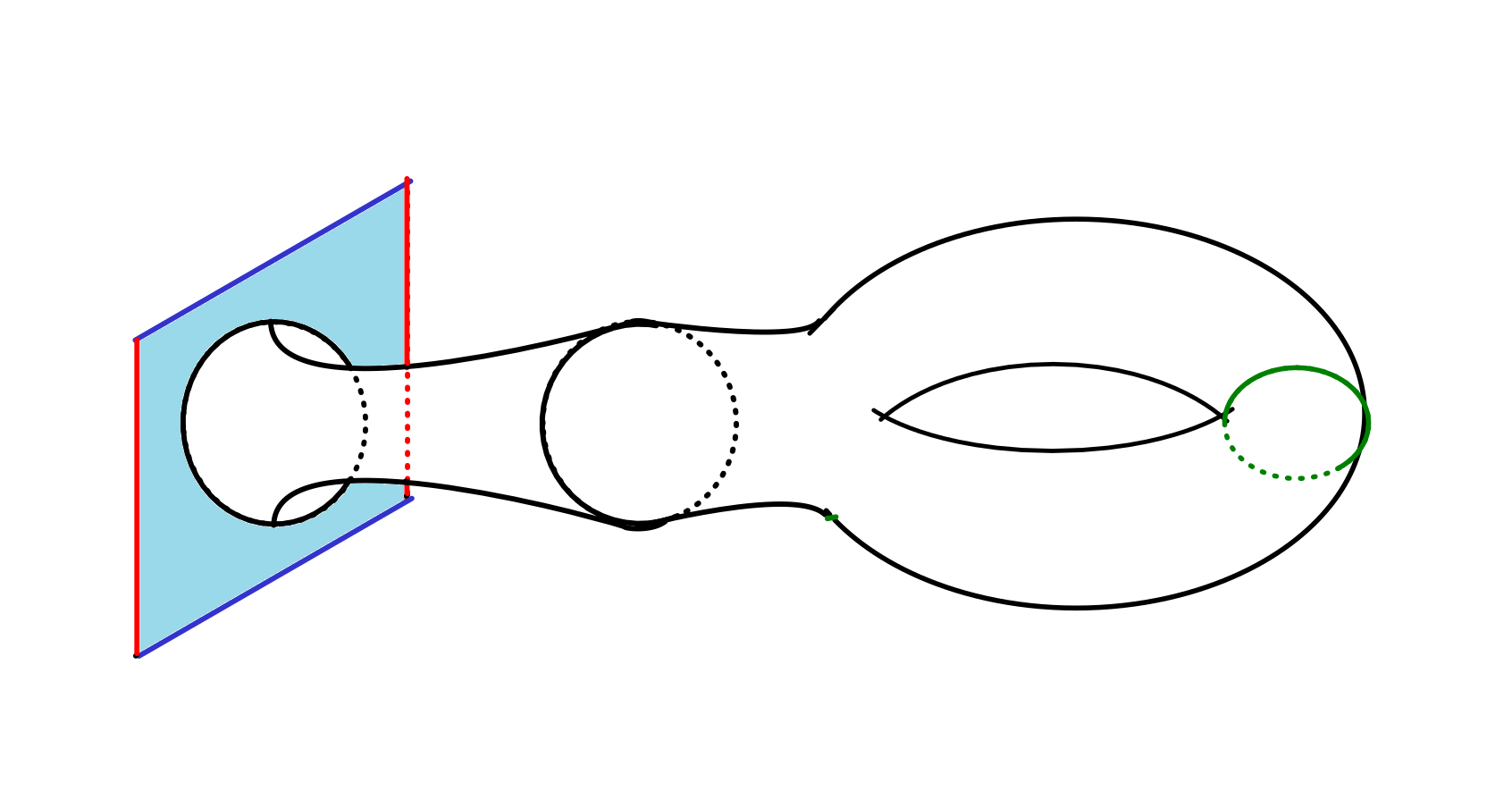}
  \label{fig:sub1}
\end{subfigure}%
\begin{subfigure}{.5\textwidth}
  \centering
\includegraphics[scale=0.13]{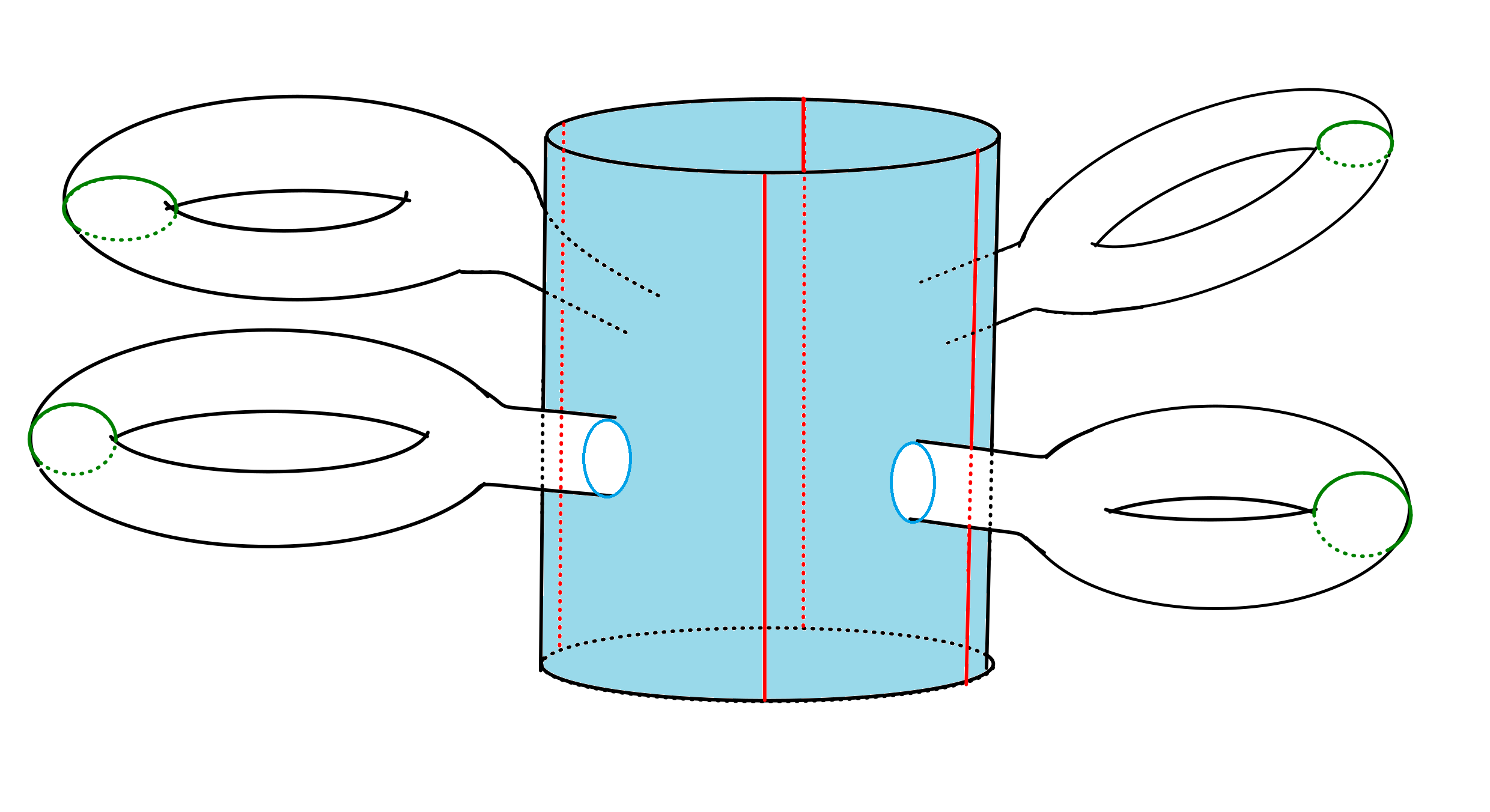}
  \label{fig:sub2}
\end{subfigure}
\caption{$\mathcal{S}$ in $B'(\ell)$ and $A$ in $S'_5$ respectively}
\label{fig:test}
\end{figure}
All that is left to do to obtain theorem $\ref{thm:main}$ is to use lemmas \ref{sigmaS_ppetit} and \ref{lowerb} in conjunction with a representation theory argument. By the the two aforementioned lemmas, we know that for $\ell$ small enough, $$\sigma_1(S_p)\not=\sigma_1(S'_p).$$
Fix this $\ell$ as such. Since $\mathbb{Z}_p$ is in the kernel of every representation of degree $1$ of $G_p$, if these representations were to act on $E_1(S_p)$ we would have that $$\sigma_1(S_p)=\sigma_1(S_p/\mathbb{Z}_p)=\sigma_1(S'_p).$$
Howerver,, we showed that it is not the case. It means that the representation of degree $1$ of $G_p$ are not acting on $E_1(S_p)$. The last option remaining is that the only representation left, of degree $p-1$, acts on $E_1(S_p)$. Thus we get that the dimension of $E_1(S_p)$ is at least $p-1$. It implies that $m_1(S_p)\geq p-1$ and theorem \ref{thm:main} follows.

\subsection*{Acknowledgments}
The author would like to thank Bruno Colbois for comments on a preliminary version of this paper. This work is part of the authors PhD thesis, under the supervision of Alexandre Girouard.

\bibliographystyle{plain}
\bibliography{ref} 

\end{document}